\documentclass[reqno]{amsart}

\usepackage{amsmath, amssymb, amsthm, epsfig}
\usepackage[hidelinks]{hyperref}
\usepackage{latexsym}
\usepackage{url}
\usepackage[mathscr]{euscript}
\usepackage{color,float}

\usepackage{url}
\usepackage[c2,nocomma]{optidef}

\usepackage{booktabs}
\usepackage{tikz}
\usetikzlibrary{calc}
 
\usepackage{setspace}

\usepackage{refcheck} 
\norefnames
\nocitenames

\def\today{\ifcase\month\or
  January\or February\or March\or April\or May\or June\or
  July\or August\or September\or October\or November\or December\fi
  \space\number\day, \number\year}

\DeclareMathOperator{\supp}{\mathrm{supp}}

\newtheorem{theorem}{Theorem}

\newtheorem{lemma}[theorem]{Lemma}

\newcommand{\z}{\mathbb{Z}}
\newcommand{\q}{\mathbb{Q}}
\renewcommand{\r}{\mathbb{R}}
\newcommand{\cp}{\mathbb{C}} 

\newcommand{\re}{{\rm Re}\,}
\newcommand{\ft}{\widehat}

\newcommand{\wt}{\widetilde}

\newcommand{\la}{\lambda}
\newcommand{\ga}{\gamma}

\newcommand{\ep}{\epsilon}

\newcommand{\si}{\sigma}

\renewcommand{\ft}{\widehat}
\DeclareRobustCommand{\rchi}{{\mathpalette\irchi\relax}}
\newcommand{\irchi}[2]{\raisebox{\depth}{$#1\chi$}} 
\newcommand{\ov}{\overline}

\begin{document}


\title[Multiplicity of nontrivial zeros of primitive $L$-functions]{Multiplicity of nontrivial zeros of primitive $L$-functions via higher-level correlations}

\author{Felipe Gon\c{c}alves}
\address{IMPA - Estrada Dona Castorina 110\\
Rio de Janeiro, Brazil} \email{goncalves@impa.br}

\author{David de Laat}
\address{Delft Institute of Applied Mathematics\\
Delft University of Technology\\ Delft, The Netherlands} \email{d.delaat@tudelft.nl}

\author{Nando Leijenhorst}
\address{Delft Institute of Applied Mathematics\\
Delft University of Technology\\ Delft, The Netherlands} \email{n.m.leijenhorst@tudelft.nl}

\date{January 8, 2024}
\allowdisplaybreaks


\begin{abstract}
We give universal bounds on the fraction of nontrivial zeros having given multiplicity for $L$-functions attached to a cuspidal automorphic representation of $\mathrm{GL}_m/\q$. For this, we apply the higher-level correlation asymptotic of Hejhal, Rudnick, and Sarnak in conjunction with semidefinite programming bounds. 
\end{abstract}

\maketitle



\section{Introduction}
Let $m\geq 1$ and let $\pi$ be an irreducible cuspidal automorphic representation of $\mathrm{GL}_m/\q$. Let $L(s,\pi)$ be its attached $L$-function (see \cite{IK} for theoretical background). Such functions generalize the classical Dirichlet $L$-functions (the case $m=1$). Let $\rho_j = 1/2+ i \gamma_j$, with $j \in \z$, be an enumeration of the nontrivial zeros of $L(s,\pi)$ repeated according to multiplicity. The Generalized Riemann Hypothesis (GRH) for $L(s,\pi)$ states that all its nontrivial zeros are aligned in the line $\re s=1/2$; that is, $\ga_j\in \r$ for all $j$. Assuming GRH, we can enumerate the ordinates of the zeros in a nondecreasing fashion
\[
\ldots \leq \ga_{-2} \leq \ga_{-1} < 0 \leq \ga_{1} \leq \ga_{2} \leq \ldots
\]
and we have
\[
N(T) := \sum_{\substack{j\geq 1 \\ \ga_j\leq T}} 1 \sim \frac{m}{2\pi} T \log T.
\]

Hejhal \cite{hejhal_triple_1994} and later on Rudnick and Sarnak \cite{MR1305671,rudnick_zeros_1996} investigated the $n$-level correlation distribution of the zeros of $L(s,\pi)$ guided by the relationship with Gaussian ensemble models in random matrix theory, pointed out by the breakthrough work of Montgomery \cite{MR0337821}. This area has a long history, and we recommend \cite{rudnick_zeros_1996}, and the references therein, for more information.

Let $a_\pi(n)$ be the coefficients such that 
\[
L(s, \pi) =\sum_{n \ge 1} \frac{a_\pi(n)}{n^s}.
\]
The results of Rudnick and Sarnack hold under the hypotheses that GRH holds for $L(s, \pi)$ and that
\[
\sum_p \frac{|a_\pi(p^k) \log p|^2}{p^k} < \infty
\]
for all $k\geq 2$. This second condition holds whenever $m \leq 3$ \cite[Proposition 2.4]{rudnick_zeros_1996}. They then show that 
\begin{align}\label{rudsarresult}
&\frac{1}{N(T)} \sum_{\substack{j_1,\dots,j_n \ge 1 \text{ distinct} \\ \gamma_{j_1},\ldots,\gamma_{j_n} \leq T}} f(\tfrac{m\log T}{2\pi } \gamma_{j_1},\dots,\tfrac{m \log T}{2\pi } \gamma_{j_n})\\\nonumber
&\quad\to \int_{\mathbb R^n} f(x) W_n(x) \delta\bigg(\frac{x_1 + \ldots + x_n}{n}\bigg) \, dx_1 \cdots dx_n
\end{align}
as $T \to \infty$, where $$W_n(x)=\det \bigg[\frac{\sin(\pi(x_i-x_j))}{\pi(x_i-x_j)}\bigg]_{i,j=1,\ldots,n}$$ is Dyson's limiting $n$-level correlation density for the Gaussian Unitary Ensemble (GUE) model, and $f$ is any \emph{admissible} test function satisfying
\begin{itemize}
\item $f\in C^\infty(\r^n)$, $f$ is symmetric, and $f(x + t(1,\dots,1)) = f(x)$ for $t \in \mathbb R$;
\item $\overline{\text{supp}}(\ft f) \subseteq \{ x \in \r^n : |x_1|+\ldots+|x_n| < 2/m\}$ (in the distributional sense);
\item $f(x) \to 0$ rapidly as $|x| \to \infty$ in the hyperplane $\sum_j  x_j = 0$.
\end{itemize}

Let $m_\rho$ denote the multiplicity of $\rho$, and define
\[
N^*(T) := \sum_{\substack{j\geq 1 \\ \ga_j\leq T}} m_{\rho_j}.
\]
Montgomery's pair correlation approach has been used to obtain bounds of the form
\begin{equation}\label{eq:pair}
N^*(T) \leq (c + o(1)) N(T),
\end{equation}
where for $m=1$ the current smallest known value for $c$ assuming the Riemann hypothesis is $1.3208$ \cite{MR1132849,MR4037496,MR0337821,MR0419378}. Let $Z_n(T)$ count the number of nontrivial zeros of  $L(s,\pi)$ with multiplicity at most $n-1$ up to height $T$:
\[
Z_n(T) = \sum_{\substack{j\geq 1 \\ \ga_j\leq T,\, m_{\rho_j} \leq  n-1 }} 1.
\]
Then $Z_2(T) \geq 2N(T) - N^*(T)$, so \eqref{eq:pair} can for instance be used to give a lower bound on the fraction of simple zeros, although the current best known bound is obtained using different techniques \cite{MR3104987,MR1132849,MR4037496,ConreyGhoshGoneksimplezeros,MR1719184}.

The goal of this paper is to use the $n$-level correlations by Hejhal, Rudnick, and Sarnak to obtain bounds analogous to \eqref{eq:pair}, and use this to obtain universal bounds on the fraction of nontrivial zeros of $L(s, \pi)$ having multiplicity at most $n-1$.

\begin{theorem}\label{main}
Let $L(s,\pi)$ be the $L$-function attached to an irreducible cuspidal automorphic representation $\pi$ of $\mathrm{GL}_m/\q$. Assuming GRH for $L(s,\pi)$ we have
\[
\liminf_{T\to\infty} \frac{Z_n(T)}{N(T)} \geq \begin{cases} 
0.9614 &\text{ if $n=3$ and $m=1$,}\\ 
0.2997 &\text{ if $n=3$ and $m=2$,}\\
0.9787 &\text{ if $n=4$ and $m=1$,}
\end{cases}
\]
where the results for $(n,m)=(3,2), (4,1)$ hold under the additional assumption that certain series of rational functions are summed correctly using Maple; see Section~\ref{sec:setup}.
\end{theorem}

As far as we know, the bounds in Theorem~\ref{main} are all new. Using a different method, it has been shown that at least $95.5\%$ of the nontrivial zeros of the zeta function are simple or double zeros \cite{ConreyGhoshGonek}. Our $(n,m) = (3,1)$ case improves on this result. The case $m=2$ is of special interest, since for this there are no previous effective bounds in the literature. For $(n,m)=(2,2)$ we show in Section~\ref{sec:nequals2} our method does not produce a positive lower bound, but as shown above for $(n,m) = (3,2)$ it does. Proving a positive proportion of zeros are simple in the case $(n,m)=(2,2)$ is a notoriously difficult problem \cite{Faveri}, and no positive lower bounds for $Z_2(T)/N(T)$ are known (in general), but there has been a great deal of interesting work on simple zeros for $GL_2$ $L$-functions (via $\Omega$-results, for instance) \cite{Booker1, Booker2, Booker3, CCM,CG, Micah, Sono}.

Montgomery's approach to finding a good value for $c$ in \eqref{eq:pair} leads to an optimization problem that is similar to the one-dimensional version of the Cohn-Elkies bound \cite{MR1973059} for the sphere packing problem. Interestingly, however, the higher-order correlation approach in this paper is very different from the higher-order correlation approach for the sphere packing problem recently introduced in \cite{cohn2022}.

In Section~\ref{sec:derivation} we set up the optimization problem to compute the above bounds, and we show how the optimal solution connects to reproducing kernels, inspired by  the approach from \cite{CCLM}. In Section~\ref{sec:nequals2} we give a simpler analytical derivation of the known optimal solution for the case $n=2$.
 
For larger values of $n$ we use a computational approach to obtain rigorous bounds, whereas in Section~\ref{sec:symrankone} we use the symmetry and the rank-$1$ structure to find the bounds by solving linear systems. We then give two approaches for finding good solutions to the optimization problem. In Section~\ref{sec:fourieropt} we use a parametrization using polynomials to find a good value for the case $(n, m) = (3,1)$ and explain why this approach becomes problematic for $n > 3$. In Section~\ref{sec:setup} we give an alternative  parametrization based on lattice shifts and use this to compute bounds for $(n,m) = (3,2), (4,1)$. 

\section{Derivation of the optimization problem}
\label{sec:derivation}

Let
\[
M_n(T) = \sum_{\substack{j_1,\dots,j_n \ge 1 \text{ distinct} \\ \gamma_{j_1} = \ldots = \gamma_{j_n} \leq T}} 1 = \sum_{\substack{j\geq  1 \\ \ga_j \leq T}}(m_{\rho_j}-1) \cdots (m_{\rho_j} - n + 1).
\]
Then $M_1(T) = N(T)$ and $M_2(T) = N^*(T) - N(T)$. For a nonnegative admissible test function $f$, we have that 
\[
\frac{M_n(T)}{N(T)} f(1, \ldots, 1) \leq \frac{1}{N(T)} \sum_{\substack{j_1,\dots,j_n \ge 1 \text{ distinct} \\ \gamma_{j_1},\ldots,\gamma_{j_n} \leq T}} f(\tfrac{m\log T}{2\pi } \gamma_{j_1},\dots,\tfrac{m \log T}{2\pi } \gamma_{j_n}),
\]
where we use that 
\[
f(1,\ldots,1) = f(\tfrac{m\log T}{2\pi } \gamma_{j_1},\dots,\tfrac{m \log T}{2\pi } \gamma_{j_n})
\]
whenever $\gamma_{j_1} = \ldots = \gamma_{j_n}$. Dividing by $f(1,\ldots, 1)$ and applying \eqref{rudsarresult} therefore gives
\[
\limsup_{T\to\infty}\frac{M_n(T)}{N(T)} \leq c_{n,m},
\]
where
\[
c_{n,m} := \inf \frac{1}{f(1,\dots,1)} \int_{\mathbb R^n} f(x) W_n(x) \delta\Big(\frac{x_1 + \ldots + x_n}{n}\Big) \, dx
\]
and where the infimum is taken over nonnegative admissible test functions $f$. 

We have
\begin{align*}
& \int_{\mathbb R^n} f(x) W_n(x) \delta\bigg(\frac{x_1 + \ldots + x_n}{n}\bigg) \, dx \\ 
& \quad = \int_{\r^{n-1}} \int_\r f(x + (x_n,\ldots,x_n,0)) W_n(x + (x_n,\ldots,x_n,0)) \\
&\quad \qquad \cdot \delta\bigg(\frac{x_1 + \ldots + x_{n-1}}{n}+x_n\bigg) \, dx_n dx_1 \cdots dx_{n-1} \\
&\quad  =  \int_{\r^{n-1}} f(x_1,\dots,x_{n-1},0)W_n(x_1,\ldots,x_{n-1},0)\, dx_1 \cdots dx_{n-1}.
 \end{align*}
With  $g(x_1,\ldots,x_{n-1})=f(x_1,\ldots,x_{n-1},0)$, we see that $\overline{\supp}(\ft g)$ is now contained in the interior of $\tfrac2m  H_{n-1}$, where 
\[
H_{n-1}=\{ x \in \mathbb R^{n-1} : |x_1| + \ldots +|x_{n-1}| +  |x_1+ \ldots +x_{n-1}| \leq 1\}.
\]
Note that $H_{n-1}\subseteq [-1/2,1/2]^{n-1}$. Now let
\[
\nu_n(g) := \int_{\mathbb R^{n-1}} g(x) W_n(x,0)  \, dx
\]
for any $g \in L^1(\r^{n-1})$ satisfying
\begin{itemize}
\item[(i)] $\overline{\supp}(\widehat g) \subseteq \frac2m H_{n-1}$;
\item[(ii)] $g$ is nonnegative;
\item[(iii)] $g(0) = 1$.
\end{itemize}
It is now obvious that $c_{n,m}\geq \inf_{g} \nu_n(g)$, where the infimum is taken over functions $g$ satisfying (i), (ii) and (iii). We show that $c_{n,m}= \inf_{g} \nu_n(g)$. Let $g \in L^1(\r^{n-1})$ satisfy the conditions above and assume further that $g$ is invariant under the group $G$ generated by
\begin{align*}\label{gsymmetry}
F_ig(x)=g(x-x_i (1,\dots,1) - x_ie_i),
\end{align*}
for $i=1,\dots,n-1$. Note that $F_i^2={\rm Id}$ and $F_iF_jF_i$ simply flips $x_i$ by $x_j$. In particular, $G$ is a finite group. Let $\varphi_\ep \in C^\infty(\r^{n-1})$ be a radial approximate identity as $\ep\to 0$, such that $\supp( \varphi_\ep)\subset \ep H_{n-1}$, $\int \varphi_\ep=1$ and $ \varphi_\ep\geq 0$. Let $g_\ep = g\ft \varphi_\ep$, and observe that since $\ft g_\ep=\ft g * \varphi_\ep$, we then have that $g_\ep\in \mathcal{S}(\r^{n-1})$ and 
\[
\supp(\ft g_\ep) \subset \left(\frac2m+\ep\right) H_{n-1}.
\]
In particular, the function $f_\ep(x_1,...,x_n)=g_\ep(a (x_1-x_n),...,a (x_{n-1}-x_n))$, with $a=1/(1+m\ep/2)$, is now an admissible test function for \eqref{rudsarresult} and we obtain
\[
c_{n,m} \leq \int_{\mathbb R^n} f_\ep(x) W_n(x) \delta\Big(\frac{x_1 + \ldots + x_n}{n}\Big) \, dx = \nu_n(g_\ep(a\cdot)) \leq \nu_n(g(a \cdot)).
\]
Taking $\ep\to 0$ we conclude that $c_{n,m}\leq \inf_{g} \nu_n(g)$ over functions $g$ satisfying (i), (ii), (iii) and  invariant under $G$. Assume now $g$ is not invariant under $G$. Since $W_n(x_1, \dots,x_{n-1},0)$ is $G$-invariant, is it clear that 
$\nu_n(g)=\nu_n(\widetilde{g})$, where
$$
\widetilde{g}(x) = \frac{1}{|G|}\sum_{\rho \in G}\rho g(x),
$$
and so $\nu_n(g)\geq c_{n,m}$.

Let $\mathrm{PW}(\Omega)$ be the Paley-Wiener space of functions $g \in L^2(\r^{n-1})$ whose Fourier transform is supported in $\Omega$. Given a finite-dimensional subspace $F$ of $\mathrm{PW}(\frac1m H_{n-1})$ and a basis $\{g_i\}$ of $F$, we can optimize over functions of the form
\begin{equation}\label{eq:sdpform}
g(x) = \sum_{i,i'} X_{i,i'} g_i(x) g_{i'}(x),
\end{equation}
where $X$ is a positive semidefinite matrix. These are integrable functions satisfying conditions (i) and (ii). Since $\nu_n(g)$ and $g(0)$ are linear in the entries of $X$, this reduces the optimization problem to the semidefinite program
\begin{mini}
  {}{\nu_n(g)}{}{}
  \label{pr:origsdp}
  \addConstraint{g(0)}{= 1}{}
  \addConstraint{X}{\succeq 0}{}.
\end{mini}
Here $X \succeq 0$ means $X$ is a positive semidefinite matrix of appropriate size. Since we optimize a linear functional over positive semidefinite matrices with affine constraints, this is a semidefinite program, which can be solved efficiently in practice. 

We have that
\[
\sum_{\substack{j\geq  1 \\ \gamma_j \le T, \, m_{\rho_j} \geq n}} 1  \leq \frac{M_n(T)}{(n-1)!},
\]
and therefore
\[
\frac{Z_n(T)}{N(T)} \geq 1 -  \frac{M_n(T)}{N(T)(n-1)!}.
\]
Thus, 
\[
\liminf_{T\to\infty} \frac{Z_n(T)}{N(T)} \geq 1-\frac{c_{n,m}}{(n-1)!},
\]
which shows we can use the problem in \eqref{pr:origsdp} to obtain a lower bound on the fraction of zeros having multiplicity at most $n-1$.

As an alternative to $M_n(T)$ we could also use the parameters 
\[
N_k(T) = \sum_{\substack{j\geq 1 \\ \ga_j\leq T}} m_{\rho_j}^{k-1},
\]
which satisfy $N(T) = N_1(T)$ and $N^*(T) = N_2(T)$. We have
\begin{equation}\label{eq:striling}
N_n(T) = \sum_{k=1}^n S(n, k) M_k(T),
\end{equation}
where $S(n,k)$ are the Stirling numbers of the second kind. Since $S(n,k) \ge 0$ for all $n$ and $k$, we can use our bounds on $c_{n,m}$, computed Section~\ref{sec:fourieropt} and \ref{sec:setup}, to obtain
\begin{equation}\label{eq:NnT}
\limsup_{T\to\infty} \frac{N_n(T)}{N(T)} \le \begin{cases} 
2.0597 &\text{ if $n=3$ and $m=1$,}\\ 
5.8984 &\text{ if $n=3$ and $m=2$,}\\ 
3.8834 &\text{ if $n=4$ and $m=1$,}
\end{cases}
\end{equation}provided the hypothesis of Theorem~\ref{main} hold. We also adapted the problem in \eqref{pr:origsdp} to bound $N_n(T)$ directly, but we were not able to improve over the results in \eqref{eq:NnT} in this way.

We end this section with a possible formal solution to the problem of computing $c_{n,m}$. Note that the following result is true for $H_{n-1}$ replaced by any compact set with a nonempty interior, with essentially the same proof.

\begin{theorem}\label{thmkernel}
For $n,m > 0$ there exists a kernel $K \colon \cp^{n-1}\times \cp^{n-1}\to\cp$ such that 
\begin{equation}\label{eq:kernel}
g(w)=\int_{\r^{n-1}} g(v)\overline{K(w,v)}\, d\nu_n(v)
\end{equation}
for all $g\in \mathrm{PW}(\frac1m H_{n-1})$ and $w\in \cp^{n-1}$.
We have
\[
c_{n,m} \leq \frac{1}{K(0,0)},
\] 
and equality is attained if every nonnegative $g\in L^1(\r^n)$ with $\supp \ft g \subseteq \frac2m H_{n-1}$ is a sum of squares in the sense that there are $g_{j,N}\in \mathrm{PW}(\tfrac1m  H_{n-1})$ such that
\begin{equation}\label{eq:sumofsquareslimit}
g(x)= \lim_{N\to\infty} \sum_{j=1}^{N}|g_{j,N}(x)|^2,
\end{equation}
with almost everywhere convergence.
\end{theorem}

\begin{proof}
We have that $0\leq W_n(x) \leq 1$ for all $x$, and $W_n(x)=0$ if and only if $x_i=x_j$ for some $i\neq j$. Using induction it can be shown that for every $\epsilon>0$ there exists $c(\epsilon) >0$ such that $W_n(x)\geq c(\epsilon)$ whenever $|x_i-x_j|\geq \epsilon$ for all $i\neq j$. Define
\[
S_\epsilon = \{x \in \r^{n-1} : |x_i|>\epsilon  \text{ and } |x_i - x_j | > \epsilon \text{ for all distinct } i,j=1,\ldots,n-1\}.
\]
Then $\nu_n(|g|^2) \leq \|g\|_2^2$ and 
\[
\nu_n(|g|^2) \geq c(\epsilon) \int_{S_\epsilon} |g(x)|^2 \, dx.
\]
Observe that the set $S_\ep$ is relatively dense, that is, there is a cube $Q$ (e.g., $Q=[-3\ep,3\ep]^{n-1}$) and $\gamma>0$ such that 
\[
\inf_{x\in \r^{b-1}} \text{Vol}(S_\ep \cap (Q+x)) \geq \gamma  \text{Vol}(Q).
\]
We could then apply the Logvinenko-Sereda Theorem \cite[p. 112]{LS}, which says there exists $b(\ep)>0$ such that 
$$
\int_{S_\epsilon} |g(x)|^2 \, dx \geq b(\ep) \|g\|_2^2,
$$
proving the norms $\nu_n(|g|^2)$ and $ \|g\|_2^2$ are equivalent. For completeness, we give a direct proof of the above inequality in our particular case.

We apply the Hardy-Littlewood-Sobolev theorem of fractional integration (see, e.g., \cite[Section V.1.2]{MR0290095}), which implies that there exists a constant $C > 0$ such that
\[
\|h(\cdot) |\cdot |^{-1/4} \|_{2} \le C \| \ft h \|_{4/3}
\]
for all $h \in L^2(\r)$. If in addition $\ov{\supp}(\ft h) \subseteq [-a,a]$, then using H\"older's inequality we obtain
\[
\|h(\cdot)|\cdot |^{-1/4} \|_{2} \le C\| \ft h \|_{4/3} \leq (2a)^{1/4} C \|\ft h \|_2 = (2a)^{1/4} C \| h \|_{2}.
\]
 Then, with $h(x_i) = \wt g(x + x_j e_i)$ and $a=1/2$, we get 
\[
\int_{\{ x : |x_i-x_j| < \epsilon\}} |g(x)|^2 \, dx \leq \epsilon^{1/2} \int_{\{ x : |x_i-x_j| < \epsilon\}} \left|\frac{g(x)}{|x_i-x_j|^{1/4}}\right|^2 \, dx \le C \epsilon^{1/2}  \|g\|_2^2.
\]
A similar procedure shows that  $\int_{\{ x : |x_i| < \epsilon\}} |g(x)|^2 \, dx \leq C \ep^{1/2} \|g\|_2^2$.
Hence, 
\begin{align*}
\int_{S_\epsilon} |g(x)|^2 \, dx 
&\geq \|g\|_2^2 - \sum_{i=1}^{n-1} \int_{\{ x : |x_i| < \epsilon\}} |g(x)|^2 \, dx - \sum_{1 \le i < j \le n} \int_{\{ x : |x_i-x_j| < \epsilon\}} |g(x)|^2 \, dx\\
&= \left(1 - \binom{n}{2} C \epsilon^{1/2}\right) \|g\|_2^2.
\end{align*}
By choosing $\epsilon > 0$ small enough the constant $1 - \binom{n}{2} C \epsilon^{1/2}$ is positive, which shows $\nu_n$ defines a norm equivalent to the $L^2$ norm. Hence, $\mathrm{PW}(\tfrac1m H_{n-1})$ is a Hilbert space with the inner product given by $\nu_n$.

Using Fourier inversion, Cauchy's inequality, and Plancherel's theorem it can be shown that the evaluation functions $f \mapsto f(x)$ on $\mathrm{PW}(\frac1m H_{n-1})$ are continuous in the $L^2$ topology, so by the Riesz representation theorem there exists a kernel $K \colon \mathbb{C}^{n-1} \times \mathbb{C}^{n-1} \to \mathbb{C}$ with $K(x,\cdot) \in L^2(\mathbb C^{n-1})$ satisfying \eqref{eq:kernel}.

To finish, observe that if $g$ satisfies conditions (i), (ii), and (iii) and condition \eqref{eq:sumofsquareslimit}, then by the monotone convergence theorem, limit \eqref{eq:sumofsquareslimit} holds also in the $L^1(\r^{n-1})$-sense, and since $\mathrm{PW}(\tfrac1m H_{n-1}) \cap L^1(\r^{n-1})$ is a reproducing kernel space, it is simple to show that we actually have uniform convergence in compact sets. Then,
\begin{align*}
1=g(0)=\lim_{N\to\infty} \sum_{j=1}^{N}|g_{j,N}(0)|^2 & = \lim_{N\to\infty} \sum_{j=1}^{N} |\nu_n(g_{j,N}(\cdot)K(0,\cdot))|^2 \\
& \leq  \lim_{N\to\infty} \sum_{j=1}^{N} \nu_n(|g_{j,N}|^2)K(0,0) = \nu_n(g)K(0,0).
\end{align*}
This shows $c_{n,m}\geq 1/K(0,0)$.  Since $g(x)=K(0,x)^2/K(0,0)^2$ satisfies conditions (i), (ii), and (iii) and $\nu_n(g) = 1/K(0,0)$, we have $c_{n,m} \leq 1/K(0,0)$, which completes the proof.
\end{proof}

It is worth mentioning that condition \eqref{eq:sumofsquareslimit} holds true, and so equality $c_{n,m}=K(0,0)^{-1}$ is true, in the one-dimensional case (Hadamard factorization), which implies this condition is true in higher dimensions if $H_{n-1}$ is replaced by a cube.  Condition \eqref{eq:sumofsquareslimit} also holds true for a ball (see \cite{Efimov}) and radial $g$, which is enough to show $c_{n,m}=K(0,0)^{-1}$. It is an open problem to prove condition \eqref{eq:sumofsquareslimit} when $H_{n-1}$ is replaced by a generic convex body.  J. Vaaler, via personal communication, proposed an interesting way of obtaining \eqref{eq:sumofsquareslimit}. He conjectures the following: For any centrally symmetric convex body $K$ there is a constant $c(K)\in(0,1)$ such that for any nonnegative $g\in {\rm PW}(K) \cap L^1$ there is $h\in {\rm PW}(K)$ such that 
\[
g\geq |h|^2 \quad \text{and} \quad \|h\|_{2}^2 \geq c(K) \|g\|_1.
\]
If this conjecture is true, iterating this procedure ones finds functions $h_1,h_2,\ldots$ such that $g\geq |h_1|^2+\ldots+|h_N|^2$ and
\[
\|h_N\|_{2}^2 \geq c(K)( \|g\|_1 - \|h_1\|_{2}^2-\ldots-\|h_{N-1}\|_{2}^2).
\]
Then it is easy to see that we must have $g=\sum_{n\geq 1} |h_n|^2$ pointwise.

\section{Analytic solution for \texorpdfstring{$n=2$}{n=2}}
\label{sec:nequals2}

In this section, we compute the kernel $K$ of Theorem \ref{thmkernel} for $n=2$ and all $m\geq 1$. We note that this kernel was already computed in \cite{CCLM} but only for $m=1$, and it relied heavily on the theory of de Branges spaces and Fourier transform techniques. In what follows, we present a direct and simple approach that avoids any such things. Moreover, if one wants to compute only $K(0,0)$, then the argument is really short.

We want to compute the reproducing kernel $K(w,z)$ such that
$$
f(w)=\int_\r f(x)\ov{K(w,x)}(1-s(x)^2) \, dx
$$
for every $w\in\cp$ and $f\in \mathrm{PW}([-\tfrac1{2m},\frac1{2m}])$, where $s(x)=\tfrac{\sin(\pi x)}{\pi x}$. Letting $k(w,z)=mK(mw,mz)$ we obtain
$$
f(w)=\int_\r f(x)\ov{k(w,x)}(1-s(m x)^2) \, dx =: \langle f | k(w,\cdot) \rangle
$$
for every $w\in\cp$ and $f\in \mathrm{PW}([-\tfrac12,\tfrac1{2}])$. In what follows we use the trigonometric identity $2\sin^2(\theta)=1-\cos(2\theta)$. Observe that if $g\in \mathrm{PW}([-1,1])$ is even then
\begin{align*}
\int_\r g(x)s(m x)^2dx  &= \int_\r \frac{g(x)-g(0)}{2(\pi m x)^2}(1-\cos(2\pi m x))dx + \frac1m g(0) \\
& = \int_\r \frac{g(x)-g(0)}{2(\pi m x)^2}dx + \frac1m g(0).
\end{align*}
Since $1-s(x)^2$ is even, a simple computation shows $k(-w,z)=k(w,-z)$, and so $k(0,x)$ is even. Note also that any reproducing kernel satisfies $k(w,z)=\ov{k(z,w)}=k(\ov z, \ov w)$. For simplicity we let
$$
k_\pm(w,z)=\frac12(k(w,z)\pm k(-w,z))
\ \ \text{ and } \ \ 
s_\pm(w,z)=\frac12(s(w-z)\pm s(w+z)).
$$
For $x,y\in\r$ we obtain
\begin{align*}
s_+(x,y)  & = \langle s_+(x,\cdot) | k_+(y,\cdot) \rangle  \\  
&= k_+(x,y) - \frac1m s(x)k(0,y)- \int_\r \frac{s_+(x,t) {k_+(t,y)}-s(x) {k(0,y)}}{2(\pi m t)^2}dt\\
& = k_+(x,y)-\frac1m s_+(x,0)k(0,y) - \int_\r s_+(x,t) \frac{ k_+(t,y)-k(0,y)}{2(\pi m t)^2}dt \\ & \quad - k(0,y)\int_\r \frac{s_+(x,t) -s(x)}{2(\pi m t)^2}dt \\
& =  k_+(x,y)-\frac1m s(x)k(0,y)  -  \frac{ k_+(x,y)-k(0,y)}{2(\pi m x)^2} \\
&\quad   - k(0,y)\frac{1-\pi x \sin(\pi x)-\cos(\pi x)}{2(\pi m x)^2} \\
& =  k_+(x,y)\bigg(1-\frac1{2(\pi m x)^2}\bigg)+k(0,y)\bigg(-\frac1m s(x)  + \frac{\pi x \sin(\pi x)+\cos(\pi x)}{2(\pi m x)^2}\bigg),
\end{align*}
where we used the identity 
$$
\int_\r \frac{s_+(x,t) -s(x)}{2(\pi m t)^2}dt = \frac{1-\pi x \sin(\pi x)-\cos(\pi x)}{2(\pi m x)^2}
$$
which can be proven via the residue theorem.
We now evaluate at $x=u_m$ with $u_m=(\pi m \sqrt{2})^{-1}$ to get
$$
k(0,y)=\frac{s_+(u_m,y)}{-\sqrt{2} \sin(\pi u_m)  + {\pi u_m \sin(\pi u_m)+\cos(\pi u_m)}}
$$
which gives
$$
K(0,0)=\frac{k(0,0)}{m}=\frac1{\frac{1}{\sqrt{2}}\cot(\tfrac1{\sqrt{2}m})-\frac{2m-1}{2m}}.
$$

We note that from this point on one can easily extract $k_{+}(x,y)$ by back substitution, if we let $a(x)=-s(x)/m   + \frac{\pi x \sin(\pi x)+\cos(\pi x)}{2(\pi m x)^2}$ then
$$
k_{+}(x,y) = \frac{s_+(x,y)-s_+(u_m,y)a(x)/a(u_m)}{1-\frac1{2(\pi m x)^2}}.
$$
For $k_{-}(x,y)$, one can do a similar trick to extract $\partial_t k(y,t)|_{t=0}$. We have
\begin{align*}
s_-(x,y)&=k_-(x,y)-\int_\r \frac{s_-(x,t)k_-(y,t)}{2(\pi m t)^2}dt \\
& = k_-(x,y)-\frac1{x}\int_\r \frac{(s_+(x,t)-\cos(\pi x)s(t))k_-(y,t)/t}{2(\pi m)^2}dt \\
& = k_-(x,y)\bigg(1-\frac1{2(\pi m x)^2}\bigg)+\frac{\cos(\pi x)}{2(\pi m)^2 x}\partial_t k(y,t)|_{t=0},
\end{align*}
where we used that $s_-(x,t)/t =s_+(x,t)-\cos(\pi x)s(t)$, hence evaluation at $x=u_m$ gives
$$
\partial_t k(y,t)|_{t=0} = \frac{  s_-(u_m,y)}{u_m\cos(\pi u_m)}.
$$
Letting $b(x)=\frac{\cos(\pi x)}{2(\pi m)^2 x}$ we get
$$
k_-(x,y) = \frac{s_-(x,y)-s_-(u_m,y)b(x)/b(u_m)}{1-\frac1{2(\pi m x)^2}}
$$
and
$$
K(x,y)= \frac1m(k_+(x/m,y/m) + k_-(x/m,y/m)).
$$

\section{Symmetry and rank-1 structure}
\label{sec:symrankone}

In this section, we show how we can use symmetry and a rank-$1$ structure to solve the problem in \eqref{pr:origsdp} efficiently.

Let $G_n$ be the symmetry group of $H_{n-1}$; that is, let $G_n$ be the group containing the matrices $A \in \r^{(n-1) \times (n-1)}$ with $A H_{n-1} = H_{n-1}$. Let 
\[
\Gamma_n = \{A^{\sf T} : A \in G_n\}.
\]

If $g$ is a function satisfying conditions (i), (ii), and (iii), then, for $\gamma \in \Gamma_n$, the function $L(\gamma) g$ defined by $L(\gamma) g(x) = g(\gamma^{-1} x)$, also satisfies these constraints and $\nu_n(g) = \nu_n(L(\gamma) g)$. It follows that the function $\overline g$ defined by
\[
\overline g(x) = \frac{1}{|\Gamma_n|} \sum_{\gamma\in\Gamma_n} L(\gamma)g(x)
\]
also satisfies these constraints and $\nu_n(\overline g) = \nu_n(g)$. Since $\overline g$ is $\Gamma_n$-invariant, this shows we may restrict to $\Gamma_n$-invariant functions.

Now assume $F$ is some $\Gamma_n$-invariant space of functions $\r^{n-1} \to \r$. Consider a complete set $\widehat \Gamma_n$ of irreducible, unitary representations of $\Gamma_n$. Let $\{g_{\pi,i,j}\}$ be a symmetry adapted system of the space $F$. By this we mean that $\{g_{\pi,i,j}\}$ is a basis of $F$ such that $H_{\pi,i} := \mathrm{span}\{g_{\pi,i,j} : j=1,\ldots,d_\pi\}$ is a $\Gamma_n$-irreducible representation with $H_{\pi,i}$ equivalent to $H_{\pi',i'}$ if and only if $\pi=\pi'$, and for each $\pi$, $i$ and $i'$ there exists a $\Gamma_n$-equivariant isomorphism $T_{\pi,i,i'} \colon H_{\pi,i} \to H_{\pi,i'}$ such that $g_{\pi,i',j} = T g_{\pi,i,j}$ for all $j$. In other words, a symmetry adapted system is a basis of $F$ according to a decomposition of $F$ into $\Gamma_n$-irreducible subspaces, where the bases of equivalent irreducibles transform in the same way under the action of $\Gamma_n$.

Then, assuming the representations of $\Gamma_n$ are of real type, we have the following block-diagonalization result: any function $g$ that is a sum of squares of functions from $F$ can be written as
\[
g(x) = \sum_\pi \sum_{i,i'} X^{(\pi)}_{i,i'} \sum_j g_{\pi,i,j}(x) g_{\pi,i',j}(x),
\]
for positive semidefinite matrices $X^{(\pi)}$. This follows from Schur's lemma as explained in \cite{MR2067190} (see \cite[Proposition~4.1]{clusteredlowranksolver} for a detailed proof). 

To generate a concrete symmetry-adapted system $\{g_{\pi,i,j}\}$ for $F$ we use an implementation of the projection algorithm as described in \cite{serre_representations_1977}. For this, first define the operators
\begin{equation}\label{eq:symmetrization}
p_{j,j'}^{(\pi)} = \frac{d_\pi}{|\Gamma|} \sum_{\gamma\in\Gamma} \pi(\gamma^{-1})_{j,j'} L(\gamma), 
\end{equation}
where $d_\pi$ is the dimension of $\pi$ and $L(\gamma)$ is the operator $L(\gamma) g(x) = g(\gamma^{-1} x)$. Then choose bases $\{g_{\pi,i,1}\}_i$ of $\mathrm{Im}(p^{(\pi)}_{1,1})$ and set $g_{\pi,i,j} = p_{j,1}^{(\pi)} g_{\pi,i,1}$.

\begin{lemma}
If $\pi$ is not the trivial representation, then $g_{\pi,i,j}(0) = 0$.
\end{lemma}
\begin{proof}
Let $p_{j,j'}^{(\pi)}$ be the operator as defined in \eqref{eq:symmetrization} and $g \in F$. Then,
\[
(p_{j,1}^{(\pi)} p_{1,1}^{(\pi)} g) (x) = \frac{d_\pi^2}{|\Gamma|^2} \sum_{\beta,\gamma\in\Gamma} \pi(\beta^{-1})_{j,1} \pi(\gamma^{-1})_{1,1} g(\beta^{-1} \gamma^{-1} x),
\]
and thus
\[
(p_{j,1}^{(\pi)} p_{1,1}^{(\pi)} g) (0) = \frac{d_\pi^2}{|\Gamma|^2} \left(\sum_{\beta\in\Gamma} \pi(\beta^{-1})_{j,1} \right)
\left(\sum_{\gamma\in\Gamma} 
\pi(\gamma^{-1})_{1,1}\right) g(0).
\]
If $\pi$ is not the trivial representation, then $\sum_{\gamma\in\Gamma} \pi(\gamma^{-1})_{j,j'} = 0$ by orthogonality of matrix coefficients. By linearity, the result then follows.
\end{proof}

This shows optimizing over functions of the form \eqref{eq:sdpform} is the same as optimizing over functions of the form
\[
g(x) = \sum_{i,i'} X_{i,i'} g_{1,i,1}(x) g_{1,i',1}(x),
\]
where $X$ is positive semidefinite and where we denote the trivial representation by $1$. That is, instead of using a $\Gamma$-invariant space of functions, we can use the much smaller space of $\Gamma$-invariant functions $g_i(x) = g_{1,i,1}(x)$. In other words, the problem in \eqref{pr:origsdp} reduces to the simpler semidefinite program
\begin{mini}
  {}{\langle A, X \rangle}{}{}
  \label{pr:sdp}
  \addConstraint{\langle b b^{\sf T}, X \rangle}{= 1}{}
  \addConstraint{X}{\succeq 0}{},
\end{mini}
where $A_{i, i'} = \nu_n(g_i g_{i'})$ and $b_i = g_{i}(0)$. Here we use the notation $\langle A, B \rangle = \mathrm{tr}(A^{\sf T} B)$.

A semidefinite program with $m$ equality constraints has an optimal solution of rank $r$ with $r(r+1)/2 \leq m$ \cite{Pataki1998OnTR}. Since the above semidefinite program has only one equality constraint, we know that there exists an optimal solution of rank $1$. In fact, the following lemma shows we can find this rank $1$ solution by solving a linear system. We can view this lemma as a finite-dimensional version of Theorem~\ref{thmkernel}.

\begin{lemma}\label{lem:simplesdp}
Let $A$ be a positive semidefinite matrix and let $x$ be a solution to the system $Ax=b$. Then  $x x^{\sf T}/(x^{\sf T} b)^2$ is an optimal solution to the semidefinite program
\begin{mini*}
  {}{\langle A, X \rangle}{}{}
  \addConstraint{\langle b b^{\sf T}, X \rangle}{= 1}{}
  \addConstraint{X}{\succeq 0}{}.
\end{mini*}
\end{lemma}

\begin{proof}
Let $X$ be feasible to the semidefinite program and consider the spectral decomposition $X = \sum_i v_i v_i^{\sf T}$.
Using Cauchy-Schwarz we have
\[
1 = b^{\sf T} X b = \sum_i (b^{\sf T} v_i)^2 = \sum_i (x^{\sf T} A v_i)^2 \leq \sum_i (x^{\sf T} A x)(v_i^{\sf T} A v_i) = 
x^{\sf T} b \langle A, X\rangle,
\]
so 
\[
\langle A, X\rangle \geq \frac{1}{x^{\sf T} b} = \left\langle A, \frac{x x^{\sf T}}{(x^{\sf T} b)^2}\right\rangle.\qedhere 
\]
\end{proof}

\section{Parametrization using polynomials}
\label{sec:fourieropt}

Here we consider the case $(n,m)=(3,1)$. Let $\rchi_{H_2}(x)$ be the indicator function of the region $H_2$ and define the functions $g_\alpha(x)$ for $\alpha = (\alpha_1,\alpha_2) \in \mathbb N_0^2$ by
\[
\widehat g_\alpha(x) = x_1^{\alpha_1} x_2^{\alpha_2} \rchi_{H_2}(x).
\]
Fix a degree $d$ and let $F = \mathrm{span}\{g_\alpha : \alpha_1 + \alpha_2 \leq d\}$. As before, we let $\{g_i\}$ be a basis for the $\Gamma_n$-invariant functions in $F$, and we use the parametrization
\[
g(x) = \sum_{i,i'} X_{i,i'} g_{i}(x) g_{i'}(x).
\]
Since each function $g_i$ is a linear combination of the functions in $F$ we have that $\widehat g_i(x) = p_i(x) \chi_{H_2}(x)$ for some polynomial $p_i$ of degree at most $d$; and as explained in Section~\ref{pr:origsdp} these polynomials $p_i$ can be computed explicitly.

To set up the linear system for solving \eqref{pr:sdp} we need to  compute explicitly how the linear functionals $g(0)$ and $\nu_3(g)$ depend on the entries of the matrix $X$. 

The region $H_2$ is the open hexagon with vertices $(1/2,0)$, $(1/2,-1/2)$, $(0,-1/2)$, $(-1/2,0)$, $(-1/2,1/2)$, $(0,1/2)$. Its symmetry group $G_3$ is the dihedral group of order $12$ with generators 
\[
r = \begin{pmatrix} 0 & -1 \\ 1 & 1\end{pmatrix} \quad \text{and} \quad s = \begin{pmatrix} 0 & 1 \\ 1 & 0\end{pmatrix}.
\]
With $C = [-1/2, 0] \times [0, 1/2]$ we have
$
H_2 = C \cup r^2 C \cup r^4 C$;
see Figure~\ref{fix:hexsquare}.

\begin{figure}
\begin{tikzpicture}[scale=2]
    \fill[gray!40] (-0.5,0) -- (0,0) -- (0,0.5) -- (-0.5,0.5) -- cycle;
    \node at (-0.25, 0.25) {$C$};
    
    \draw[ultra thin, ->] (-0.7,0) -- (0.7,0) node[right] {$x_1$};
    \draw[ultra thin,->] (0,-0.7) -- (0,0.7) node[above] {$x_2$};

    \draw[thick] (0.5,0) -- (0.5,-0.5) -- (0,-0.5) -- (-0.5,0) -- (-0.5,0.5) -- (0,0.5) -- cycle;
\end{tikzpicture}
\caption{The hexagon $H_2$ and the square $C$ (shaded).}
\label{fix:hexsquare}
\end{figure}

Since $g_{i}$ is $\Gamma_3$-invariant, the polynomial $p_{i}(x)$ is $G_3$-invariant, so
\begin{align*}
g_i(0) & = \int_{H_2} p_{i}(x) \, dx = \int_C \big(p_{i}(x) + p_{i}(r^2x) + p_{i}(r^4x)\big) \,dx = 3\int_C p_i(x) \, dx \\
& = 3\int_0^{1/2}\int_{-1/2}^0 p_{i}(x_1, x_2)  \, dx_1 dx_2,
\end{align*}
which can be computed explicitly. This shows how we can write
\[
g(0) = \sum_{i,i'} X_{i,i'} g_{i}(0) g_{i'}(0)
\] 
explicitly as a linear combination of the entries of the matrix $X$.

As computed in Hejhal \cite{hejhal_triple_1994}, we have
\begin{align*}
\nu_3(g) &= \int_{\r^{2}} g(x_1,x_2)W_3(x_1,x_2,0) \, dx_1 dx_2\\
&= 2 + \widehat g(0) + 6 \int_0^1 \widehat g(x,0) (x-1) \, dx - 12 \int_0^1 \int_{-x_2}^0 \widehat g(x_1,x_2)x_2 \, dx_1  dx_2.
\end{align*}

To compute this explicitly, note that 
\[
\widehat g(x) = \sum_{i,i'} X_{i,i'} (\widehat g_{i} * \widehat g_{i'})(x).
\]
The term $\widehat g(0)$ can be computed similarly to how $g_{i}(0)$ is computed above, since
\[
\widehat g(0) = \sum_{i,i'} X_{i,i'} \int_{H_2} p_i(x) p_{i'}(x) \, dx,
\]
where we used that $p_{i'}(-x) = p_{i'}(x)$.

Since $\widehat g_i$ and $\widehat g_{i'}$ have bounded support, the convolution $(\widehat g_i *\widehat g_{i'})(x)$ is an integral over a bounded region whose shape depends on $x$. 
For the third term, symmetry gives
\[
\int_0^1 \widehat g(x,0) (x-1) \, dx = \int_0^1 \widehat g(x, -x) (x-1) \, dx.
\]
Figure~\ref{fig:convolution} shows the shape of the integration region used to compute $\widehat g(x,-x)$ for different values of $x$. This gives
\[
\int_0^1 \widehat g(x,-x)(x-1)\, dx = \sum_{i,i} X_{i,i'}  \int_0^1 \mathcal J_x \,p_{i}(u_1,u_2)p_{i'}(u_1+x,u_2-x) (x-1) \, du_1 du_2 dx,
\]
where
\[
\mathcal J_x = \begin{cases}
\int_{-\frac{1}{2}+x}^0 \int_{-\frac{1}{2}-u_2}^{\frac{1}{2}-x}+ \int_{0}^x \int_{-\frac{1}{2}}^{\frac{1}{2}-x} + \int_{x}^{\frac{1}{2}}\int_{-\frac{1}{2}}^{\frac{1}{2}-u_2} & 0 \le x \le \frac{1}{2}, \\
\int_{x-\frac{1}{2}}^{\frac{1}{2}} \int_{-\frac{1}{2}}^{\frac{1}{2}-x} & \frac{1}{2} \le x \le 1.
\end{cases}
\]

\begin{figure}
    \centering
    \begin{tikzpicture}

    \node (center2) at (0.75,-0.75) {};

    \begin{scope}
        \clip (1,0) -- (0,1) -- (-1,1) -- (-1,0) -- (0,-1) -- (1,-1) -- cycle;
        \clip ($(1,0)+(center2)$) -- ($(0,1)+(center2)$) -- ($(-1,1)+(center2)$) -- ($(-1,0)+(center2)$) -- ($(0,-1)+(center2)$) -- ($(1,-1)+(center2)$) -- cycle;
        \fill[color=gray!40] (-2,-2) rectangle (2,2);
    \end{scope}

    \draw[ultra thin, ->] (-1.4,0) -- (2.2,0) node[right] {$u_1$};
    \draw[ultra thin,->] (0,-2.2) -- (0,1.4) node[above] {$u_2$};
   
        \draw[thick] (1,0) -- (0,1) -- (-1,1) -- (-1,0) -- (0,-1) -- (1,-1) -- cycle;
        \draw[thick] ($(1,0)+(center2)$) -- ($(0,1)+(center2)$) -- ($(-1,1)+(center2)$) -- ($(-1,0)+(center2)$) -- ($(0,-1)+(center2)$) -- ($(1,-1)+(center2)$) -- cycle;
    
        \begin{scope}[shift = {(5,0)}]
        \node (center3) at (1.25,-1.25) {};

        \begin{scope}
            \clip (1,0) -- (0,1) -- (-1,1) -- (-1,0) -- (0,-1) -- (1,-1) -- cycle;
            \clip ($(1,0)+(center3)$) -- ($(0,1)+(center3)$) -- ($(-1,1)+(center3)$) -- ($(-1,0)+(center3)$) -- ($(0,-1)+(center3)$) -- ($(1,-1)+(center3)$) -- cycle;
            \fill[color=gray!40] (-2,-2) rectangle (2,2);
        \end{scope}

        \draw[ultra thin, ->] (-1.4,0) -- (2.4,0) node[right] {$u_1$};
        \draw[ultra thin,->] (0,-2.4) -- (0,1.4) node[above] {$u_2$};

        \draw[thick] (1,0) -- (0,1) -- (-1,1) -- (-1,0) -- (0,-1) -- (1,-1) -- cycle;
        \draw[thick] ($(1,0)+(center3)$) -- ($(0,1)+(center3)$) -- ($(-1,1)+(center3)$) -- ($(-1,0)+(center3)$) -- ($(0,-1)+(center3)$) -- ($(1,-1)+(center3)$) -- cycle;
    \end{scope}
    \end{tikzpicture}
    \caption{The possible shapes of the intersection of the supports of $g_i(u_1, u_2)$ and $g_{i'}(u_1+x, u_2-x)$ for $0 \leq x \leq 1/2$ (left) and $1/2 \leq x \leq 1$ (right).}
    \label{fig:convolution}
\end{figure}

Similarly, we split the convolution integral for the fourth term into the 4 regions $R_1=\{-1\leq -x_2 \leq x_1 \leq -\frac{1}{2}\}$, $R_2=\{-\frac{1}{2} \leq \frac{1}{2}-x_2 \le x_1 \le 0\}$, $R_3=\{0\leq  -2x_1 \le x_2 \le \frac{1}{2}-x_1 \leq 1\}$ and $R_4=\{ 0\leq -x_1 \le x_2 \le -2x_1 \leq 1\}$.
Then the fourth term can be computed as
\[
\int_0^1 \int_{-x_2}^0 \widehat g(x_1,x_2) x_2\,dx_1dx_2 =  \sum_{i,i'} X_{i,i'} \int_0^1 \int_{-x_2}^0 \mathcal I_x \,p_{i}(u) p_{i'}(u-x) x_2 \,du_1 du_2 dx_1 dx_2,
\]
where
\[
\mathcal I_x = \begin{cases} 
\int_{x_2-\frac{1}{2}}^{\frac{1}{2}} \int_{-\frac{1}{2}}^{x_1+\frac{1}{2}} & (x_1,x_2)\in R_1,\\
\int_{x_2-\frac{1}{2}}^{\frac{1}{2}} \int_{-\frac{1}{2}-u_2+x_1+x_2}^{\frac{1}{2}-u_2} &  (x_1,x_2)\in R_2 ,\\
\int_{x_1+x_2}^{\frac{1}{2}}\int_{-\frac{1}{2}}^{\frac{1}{2}-u_2}+ \int_{-x_1}^{x_1+x_2}\int_{x_1+x_2-u_2-\frac{1}{2}}^{\frac{1}{2}-u_2} + \int_{x_2-\frac{1}{2}}^{-x_1}\int_{x_1+x_2-u_2-\frac{1}{2}}^{x_1+\frac{1}{2}} & (x_1,x_2)\in R_3,\\
\int_{-x_1}^{\frac{1}{2}} \int_{-\frac{1}{2}}^{\frac{1}{2}-u_2}+ \int_{x_1+x_2}^{-x_1}\int_{-\frac{1}{2}}^{x_1+\frac{1}{2}} + \int_{x_2-\frac{1}{2}}^{x_1+x_2}\int_{x_1+x_2-u_2-\frac{1}{2}}^{x_1+\frac{1}{2}}& (x_1,x_2)\in R_4.
\end{cases}
\]

In Table~\ref{table:firstbounds} we show the bounds on $c_{3,1}$ obtained with this parametrization for degree~$d$. This proves the $(n, m) = (3,1)$ case of Theorem~\ref{main}. As ancillary files to the arXiv version of this paper we give the GAP source code to compute the polynomials $p_i$ in exact arithmetic and a Julia file to compute the above bound rigorously by solving the linear system from Section~\ref{sec:symrankone} in ball arithmetic.  

\begin{table}
\caption{Upper bounds on the best possible value for $c_{3,1}$ obtained using a parametrization using polynomials of degree $d$.}
\label{table:firstbounds}
\begin{center}
\begin{tabular}{cc}
\toprule
$d$ & $c_{3,1}$\\
\midrule
$10$ & $0.077516654$ \\
$20$ & $0.077222625$ \\
$30$ & $0.077206761$ \\
$40$ & $0.077200000$ \\
$50$ & $0.077198398$ \\
$60$ & $0.077197284$ \\
\bottomrule
\end{tabular}
\end{center}
\end{table}

\section{Parametrization using shifts}\label{sec:setup}

Now we consider the functions $g_\lambda$ obtained by shifting the Fourier transform of the indicator function of $\frac{1}{m}H_{n-1}$ by $\lambda \in m\z^{n-1}$:
\[
g_\lambda(x) = \widehat{\rchi_{\frac1m H_{n-1}}}(x - \lambda) = \frac{1}{m^{n-1}}\widehat{\rchi_{H_{n-1}}}\left(\frac1m(x-\lambda)\right). 
\]
Then $g_\lambda(x)$ is the Fourier transform of $\rchi_{\frac1m H_{n-1}}(\cdot) e^{-2\pi i \langle \lambda, \cdot \rangle}$, so $g_\lambda(\gamma x) = g_{\gamma^{-1} \lambda}(x)$ for all $\gamma$ in the group $\Gamma_n$ as defined in Section~\ref{sec:symrankone}. 

We observe that for $n=2$, the functions $g_\la$ for $\lambda \in \Lambda=m\z^{n-1}$ form an orthogonal basis of $\mathrm{PW}(\frac{1}{m}H_{n-1})$ with respect to the Lebesgue $L^2$-norm. This, however is not true for $n \ge 3$. For $n=3$, a classical result of Venkov \cite{venkov} shows that it is an orthogonal basis if we take $\Lambda$ as the dual lattice of 
\[
\{(u/2+v,u/2-v/2):u,v\in \frac{1}{m}\z\}.
\]
For $n\geq 4$, a recent result \cite{greenfeld,lev} implies that no set of translations $\Lambda$ exists for which $\{g_\lambda\}_{\lambda\in\Lambda}$ is an orthogonal basis of $\mathrm{PW}(\frac{1}{m}H_{n-1})$, otherwise $H_{n-1}$ would need to be centrally symmetric, have only centrally symmetric faces and tile the space by translations, but none of these properties hold. In any case, in practice, using $\Lambda=m\z^{n-1}$ works reasonably well since $\{g_\lambda\}_{\lambda\in\Lambda}$ is an orthogonal basis of $\mathrm{PW}(\frac{1}{m}[-1/2,1/2]^{n-1})$ which contains $\mathrm{PW}(\frac{1}{m}H_{n-1})$.

Let $F$ be a finite-dimensional, $\Gamma_n$-invariant subspace of $\mathrm{span}\{g_\lambda : \lambda \in m\z^{n-1}\}$, and again use the parametrization
\[
g(x) = \sum_{i,i'} X_{i,i'} g_{i}(x) g_{i'}(x),
\]
where $\{g_i\}$ is a basis for the $\Gamma_n$-invariant functions in $F$.

To set up the linear system from Section~\ref{sec:symrankone} we need to compute the matrix $A_{i,i'} = \nu_n(g_{i}g_{i'})$ and the vector $b_i = g_{i}(0)$. By Lemma~\ref{lem:simplesdp} the optimal function then is given by 
\[
g(x) = \frac{(\sum_i c_i g_{i}(x))^2}{(c^{\sf T} b)^2}
\]
with $c$ a solution to $Ac = b$. 

In Appendix~\ref{sec:fourier} we calculate the Fourier transform of $\rchi_{H_{n-1}}(x)$ for general $n$. This function has removable singularities whenever $x_i= 0$ for some $i$ or $x_i=x_j$ for some $i \neq j$. The function $g_i$ is a linear combination of shifts of the Fourier transform, and thus can also have removable singularities. To compute $g_i(0)$ we consider each term individually, check whether the denominator evaluates to zero and if it does, we use repeated applications of L'Hôpital's rule.

To compute the entries $A_{i,i'} = \nu_n(g_{i}g_{i'})$, we need to compute integrals of the form $\int_{\r^{n-1}} f(x) \, dx$, where $f(x) = g_i(x)g_{i'}(x)W_n(x,0)$. To simplify the computations we change variables to $m^{-1}x$. Since the Fourier transform of $g_{i}(m x)$ is supported in $[-1/2,1/2]^{n-1}$ and the Fourier transform of $W_n(mx,0)$ is supported in $[-m,m]^{n-1}$, the Fourier transform of $f$ is supported in $[-(m+1),(m+1)]^{n-1}$. By Poisson summation, we then have
\[
\int_{\r^{n-1}} f(x) \, dx
 =  \widehat f(0)
= \sum_{k \in (m+1)\z^{n-1}} \widehat f(k)
=
\frac{1}{(m+1)^{n-1}}\sum_{k \in \frac{1}{m+1}\z^{n-1}} f(k).
\]

By again using Poisson summation, we have
\begin{align*}
\sum_{k \in \frac{1}{m+1}\z^{n-1}} f(k)
&= (m+1)^{n-1}\widehat f(0) \\
&=(m+1)^{n-1} \sum_{k \in (m+1)\z^{n-1}} \widehat f(k) e^{2\pi i \langle k, s \rangle}\\
&= \sum_{k \in \frac{1}{(m+1)}\z^{n-1}} f(k+s)  
\end{align*}
for any $s \in \r^{n-1}$. To avoid evaluation of $f$ on a removable singularity, we set 
\[
s = \left(\frac{1}{m(m+1)n}, \frac{2}{m(m+1)n}, \ldots, \frac{n-1}{m(m+1)n}\right),
\]
and compute the integral as
\[
\int_{\r^{n-1}} f(x) \, dx = \frac{1}{(m+1)^{n-1}}\sum_{k \in \frac{1}{(m+1)}\z^{n-1}} f(k+s).
\]

Because of the sines and cosines in the formula of $\widehat \rchi_{H_{n-1}}$ (see Appendix~\ref{sec:fourier}), each series $\sum_k f(k+s)$ can be split into $(2(m+1))^{n-1}$ series of rational functions over the lattice $\z^{n-1}$, one for each value appearing for the sines and cosines. 

Computing these series exactly is the computational bottleneck in our approach. Maple is the only system among the computer algebra systems we tried in which we could successfully compute the sums of all the series, and for this we needed several ad hoc tricks. For instance, for series over $\z^3$, we first compute the outer series over $\z$, and then use partial fraction decomposition to simplify the terms before summing over $\z^2$. Here a complicating factor was that after partial fraction decomposition the series of some terms no longer converge, which meant we needed to recombine certain terms in ad hoc ways before summing over $\z^2$.

Unfortunately, Maple 2022.2 contains a bug where it acts nondeterministically and sometimes gives a completely wrong answer. We, therefore, check that the sum computed with Maple deviates at most $0.1\%$ from the sum we obtain by truncating the series to the box $[-C/(m+1), C/(m+1)]^{n-1}$, where for $n=3$ we use $C=400$ (or $C=1300$ for the cases where this is not sufficient) and for $n=4$ we use $C = 25$. We give text files containing the series in Maple syntax and  for each series an interval containing the correct sum given by decimal expressions of a midpoint and a radius. We also give a Julia file with the code to obtain the rigorous bounds using the data in these text files.

For $(n,m)=(3,1)$ and $(n,m)= (3,2)$, Table~\ref{table:secondbounds} gives the results for the subspaces $F = \mathrm{span}\{ g_\lambda : \lambda \in S_d \}$, where 
\[
S_d = \Gamma_n \{ m\lambda \in \z^{n-1} : \|\lambda\|_1 \leq d\}.
\]
This proves the case $(3,2)$ of Theorem~\ref{main} (the case $(3,1)$ was already done in Section~\ref{sec:fourieropt}). For $(n,m) = (4,1)$, we only compute the series exactly for $\lambda = 0$ (partly because of the aforementioned Maple bug), which gives $c_{4,1} \leq 447/3500 \approx 0.1277$ as used in Theorem~\ref{main}. Numerically we estimate that by using more values of $\lambda$ this can be improved to  $c_{4,1} \leq 0.026$. This would improve the lower bound in Theorem~\ref{main} from $0.9787$ to $0.9956$.

\begin{table}[ht]
\caption{Upper bounds on the best possible value for $c_{n,m}$ obtained through a parametrization using shifts $S_d^2$.} 
\label{table:secondbounds}
\begin{center}
\begin{tabular}{ccc}
\toprule
$d$ & $c_{3,1}$ & $c_{3,2}$ \\
\midrule
0      & 0.144444445 & 1.488888889\\
1      & 0.077710979 & 1.401604735\\
2      & 0.077580416 & 1.401343568 \\
3      & 0.077261926 & 1.400616000 \\
4      & 0.077247720 & 1.400581457 \\
5      & 0.077213324 & 1.400506625 \\
\bottomrule
\end{tabular}
\end{center}
\end{table}

\begin{appendix}

\section{The Fourier transform of \texorpdfstring{$\rchi_{H_{n-1}}$}{Hn-1}}
\label{sec:fourier}

To compute the Fourier transform of the indicator function of ${H_{n-1}}$ we use the Fourier transform of the indicator function of the simplex $S_n=\{x\in \r^{n} : x_j\geq 0, \, x_1 + \ldots + x_n\leq 1\}$ as computed in \cite[Eq. 18]{borda_lattice_points}:
\[
\ft \rchi_{S_n}(y)= \frac{(-1)^{n+1}}{(2\pi i)^n}\sum_{j=1}^n \frac{1-e^{-2\pi i y_j}}{y_j \prod_{i\neq j}(y_j-y_i)}.
\]
The Fourier transform of the indicator function of the cross-polytope $C_n=\{x\in \r^{n} : |x_1|+\ldots+|x_n|\leq 1\}$ is given by
\begin{align*}
\ft \rchi_{C_n}(y) & = \sum_{\si \in \{\pm 1\}^n} \ft \rchi_{S_n}(\si_1 y_1,\ldots,\si_n y_n).
\end{align*}
If we consider each term in $\ft \rchi_{S_n}(y)$ individually, then it is simple to deduce that
$$
\sum_{\si \in \{\pm 1\}^n} \frac{1-e^{-2\pi i \si_j y_j}}{\si_j y_j \prod_{i\neq j}(\si_j y_j-\si_i y_i)} = \sum_{\si\in \{\pm 1\}} (2\si y_j)^{n-1}\frac{1-e^{-2\pi i \si y_j}}{\si y_j \prod_{i\neq j}(y_j^2-y_i^2)},
$$
while the latter depends on the parity of $n$. This implies that
\[
\ft \rchi_{C_n}(y) = \begin{cases} \pi^{-n}(-1)^{(n-1)/2}\sum_{j=1}^n \frac{y_j^{n-2}\sin(2\pi y_j)}{\prod_{i\neq j}(y_j^2-y_i^2)} & \text{ if $n$ is odd, and}\\ 
\pi^{-n}(-1)^{n/2-1}\sum_{j=1}^n \frac{2y_j^{n-2}\sin(\pi y_j)^2}{\prod_{i\neq j}(y_j^2-y_i^2)} & \text{ if $n$ is even.}\end{cases}
\]
Finally, notice that $H_{n-1}$ can be identified with $C_n \cap \{x : x_1+\ldots+x_n=0\}$ and thus one can rigorously justify that 
$$
\ft \rchi_{H_{n-1}}(y_1,\dots,y_{n-1})=\lim_{T\to\infty} \int_{-T}^T \ft \rchi_{C_n}(y_1+t,\dots,y_{n-1}+t,t)\, dt 
$$
The function under the integral sign above only has simple poles at $t=-\tfrac{y_j+y_k}2$ for $k\neq j$ and $t=-y_j/2$ (for generic $y\in \r^{n-1}$). A routine application of the residue theorem and a grouping of terms shows that $\ft \rchi_{H_{n-1}}(y)$ is equal to
\[
\frac{(-1)^{ (n-1)/2 }}{2^{n-2}\pi^{n-1}}\bigg[-\sum_{1 \leq j < k \leq n-1} \frac{(y_j-y_k)^{n-3}\cos(\pi(y_j-y_k))}{y_j y_k\prod_{i\neq j,k}(y_j-y_i)(y_k-y_i)} + \sum_{j=1}^{n-1} \frac{y_j^{n-3}\cos(\pi y_j)}{\prod_{i\neq j}(y_j-y_i)y_i}\bigg]
\]
if $n$ is odd, and equal to
\[
\frac{(-1)^{n/2-1}}{2^{n-2}\pi^{n-1}}\bigg[\sum_{1 \leq j < k \leq n-1} \frac{(y_j-y_k)^{n-3}\sin(\pi(y_j-y_k))}{y_j y_k\prod_{i\neq j,k}(y_j-y_i)(y_k-y_i)} + \sum_{j=1}^{n-1} \frac{y_j^{n-3}\sin(\pi y_j)}{\prod_{i\neq j}(y_j-y_i)y_i}\bigg]
\]
if $n$ is even.

\end{appendix}

\section*{Acknowledgements}

We are deeply thankful of the anonymous referees for valuable suggestions on the manuscript and additional references.

\providecommand{\bysame}{\leavevmode\hbox to3em{\hrulefill}\thinspace}
\providecommand{\MR}{\relax\ifhmode\unskip\space\fi MR }
\providecommand{\MRhref}[2]{%
  \href{http://www.ams.org/mathscinet-getitem?mr=#1}{#2}
}
\providecommand{\href}[2]{#2}


\begin{thebibliography}{40}


\bibitem{Booker1} A.R. Booker, Simple zeros of degree 2 {$L$}-functions, J. Eur. Math. Soc. (JEMS) 18 (2016), no.4, 813-823.

\bibitem{Booker2}  A.R. Booker, P. J. Cho and M.Kim, Simple zeros of automorphic {$L$}-functions, Compos. Math.155 (2019), no. 6, 1224-1243.

\bibitem{Booker3} A.R. Booker, M. Milinovich and Micah B. and N. Ng, Quantitative estimates for simple zeros of {$L$}-functions, Mathematika 65 (2019), no. 2, 375-399.

\bibitem{borda_lattice_points}
B. Borda, {Lattice points in algebraic cross-polytopes and simplices},
  Discrete \& Computational Geometry \textbf{60} (2018), 145--169.

\bibitem{MR3104987}
H.~M. Bui and D.~R. Heath-Brown, {On simple zeros of the {R}iemann
  zeta-function}, Bull. Lond. Math. Soc. \textbf{45} (2013), no.~5, 953--961.
  \MR{3104987}

\bibitem{CCLM}
E. Carneiro, V. Chandee, F. Littmann, and M.~B. Milinovich, {Hilbert
  spaces and the pair correlation of zeros of the {R}iemann zeta-function}, J.
  Reine Angew. Math. \textbf{725} (2017), 143--182. \MR{3630120}
  
\bibitem{CCM} Carneiro, Chirre, and Milinovich, Hilbert spaces and low-lying zeros of {$L$}-functions, Adv. Math. 410 (2022), No. 108748, 48 pp.

\bibitem{MR1132849}
A.~Y. Cheer and D.~A. Goldston, {Simple zeros of the {R}iemann
  zeta-function}, Proc. Amer. Math. Soc. \textbf{118} (1993), no.~2, 365--372.

\bibitem{MR4037496}
A. Chirre, F. Gon\c{c}alves, and D. de~Laat, {Pair correlation estimates
  for the zeros of the zeta function via semidefinite programming}, Adv. Math.
  \textbf{361} (2020), 106926, 22.

\bibitem{cohn2022}
H. Cohn, D. de~Laat, and A. Salmon, {Three-point bounds for sphere
  packing}, preprint, 2022. arXiv:2206.15373

\bibitem{MR1973059}
H. Cohn and N. Elkies, {New upper bounds on sphere packings. {I}}, Ann. of
  Math. (2) \textbf{157} (2003), no.~2, 689--714. \MR{1973059}

\bibitem{ConreyGhoshGonek}
J.\ Conrey, A.\ Ghosh, S.\ Gonek, {Mean values of the Riemann zeta-function with application to the distribution of zeros}, Number Theory, Trace Formulas And Discrete Groups (Oslo, 1987) (1989), 185--199.

\bibitem{ConreyGhoshGoneksimplezeros}
J.\ Conrey, A.\ Ghosh, S.\ Gonek, {Simple zeros of the Riemann zeta-function}, Proc. London Math. Soc. (3) \textbf{76} (1998), no.~3, 497--522.

\bibitem{CG}
J. Conrey and A. Ghosh, Simple zeros of the Ramanujan $\tau$-Dirichlet series, Invent. Math. 94, 403-419 (1988).

\bibitem{clusteredlowranksolver}
D. de~Laat and N. Leijenhorst, {Solving clustered low-rank semidefinite
  programs arising from polynomial optimization}, preprint, 2022. arXiv:2202.12077

\bibitem{Efimov}
A.~V. Efimov, {An analog of Rudin's theorem for continuous radial positive
  definite functions of several variables}, Proceedings of the Steklov
  Institute of Mathematics (Supplementary issues) \textbf{284} (2014), no.~1,
  79--86.
 
\bibitem{Faveri} A. de Faveri, Simple zeros of $GL(2)$ $L$-functions, preprint at arxiv.org/abs/2109.15311.

\bibitem{MR2067190}
K. Gatermann and P.~A. Parrilo, {Symmetry groups, semidefinite programs,
  and sums of squares}, J. Pure Appl. Algebra \textbf{192} (2004), no.~1-3,
  95--128. \MR{2067190}

\bibitem{MR1719184}
D.~A. Goldston, S.~M. Gonek, A.~E. \"{O}zl\"{u}k, and C. Snyder, {On the
  pair correlation of zeros of the {R}iemann zeta-function}, Proc. London Math.
  Soc. (3) \textbf{80} (2000), no.~1, 31--49. \MR{1719184}

\bibitem{greenfeld}
R. Greenfeld and N. Lev, {Fuglede's spectral set conjecture for convex
  polytopes}, Analysis and PDE (2017), no.~10, 1497--1538.

\bibitem{LS}
V.~P. Havin and B. Joricke, {he uncertainty principle in harmonic analy-
  sis}, Springer-Verlag, Berlin Heidelberg (1994).

\bibitem{hejhal_triple_1994}
D.~A. Hejhal, {On the triple correlation of zeros of the zeta function},
  International Mathematics Research Notices \textbf{1994} (1994), no.~7,
  293--302.

\bibitem{IK}
H. Iwaniec and E. Kowalski, {Analytic number theory}, American
  Mathematical Society Colloquium Publications, AMS, Providence, RI \textbf{53}
  (2004).

\bibitem{lev}
N. Lev and M. Matolcsi, {The fuglede conjecture for convex domains is true
  in all dimensions}, Acta Mathematica (2022), no.~228, 385--420.
  
  
\bibitem{Micah}
M. Milinovich and N. Ng, Simple zeros of modular $L$-functions, Proc. London Math. Soc. (3) 109, 1465-1506 (2014).

\bibitem{MR0337821}
H.~L. Montgomery, {The pair correlation of zeros of the zeta function},
  Analytic number theory ({P}roc. {S}ympos. {P}ure {M}ath., {V}ol. {XXIV},
  {S}t. {L}ouis {U}niv., {S}t. {L}ouis, {M}o., 1972), 1973, pp.~181--193.
  \MR{0337821}

\bibitem{MR0419378}
H.~L. Montgomery, {Distribution of the zeros of the {R}iemann zeta
  function}, Proceedings of the {I}nternational {C}ongress of {M}athematicians
  ({V}ancouver, {B}. {C}., 1974), {V}ol. 1, 1975, pp.~379--381. \MR{0419378}

\bibitem{Pataki1998OnTR}
G. Pataki, {On the rank of extreme matrices in semidefinite programs and
  the multiplicity of optimal eigenvalues}, Math. Oper. Res. \textbf{23}
  (1998), 339--358.

\bibitem{MR1305671}
Z. Rudnick and P. Sarnak, {The {$n$}-level correlations of zeros of the
  zeta function}, C. R. Acad. Sci. Paris S\'{e}r. I Math. \textbf{319} (1994),
  no.~10, 1027--1032. \MR{1305671}

\bibitem{rudnick_zeros_1996}
Z. Rudnick and P. Sarnak, {Zeros of principal {$L$}-functions and random matrix theory},
  Duke Math. J. \textbf{81} (1996), no.~2, 269--322, A celebration of John F.
  Nash, Jr. \MR{1395406}

\bibitem{serre_representations_1977}
J.-P. Serre, {Linear representations of finite groups}, Graduate Texts in
  Mathematics, Vol. 42, Springer-Verlag, New York-Heidelberg, 1977, Translated
  from the second French edition by Leonard L. Scott. \MR{0450380}
  
\bibitem{Sono} K. Sono, A note on simple zeros of primitive Dirichlet $L$-functions, Bull. Aust. Math. Soc. 93 (2016), no.1, 19-30.

\bibitem{MR0290095}
E.~M. Stein, {Singular integrals and differentiability properties of
  functions}, Princeton Mathematical Series, No. 30, Princeton University
  Press, Princeton, N.J., 1970. \MR{0290095}

\bibitem{venkov}
B. Venkov, {On a class of euclidean polyhedra (russian)}, Vestnik
  Leningrad. Univ. Ser. Mat. Fiz. Him. (1954), no.~9, 11--31.

\end{thebibliography}
\end{document}